\documentclass[12pt, a4paper]{amsart}

\usepackage[top=1in, bottom=1in, left=1in, right=1in]{geometry}
\usepackage{amssymb}
\usepackage{amsfonts}
\usepackage{amsthm}
\usepackage{amsmath}

\newtheorem{theorem}{Theorem}
\newtheorem{lemma}{Lemma}

\theoremstyle{definition}
\newtheorem{remark}{Remark}
   
\newcommand{\R}{\mathbb R}

\newcommand{\op}[1]{\operatorname{{#1}}}

\newcommand{\mc}{\mathcal}
\newcommand{\dis}{\displaystyle}

\newcommand{\eps}{\varepsilon}

\allowdisplaybreaks

\title{Differential Harnack Estimates For Fisher's Equation}
\author{Xiaodong Cao}
\address{Department of Mathematics,
 Cornell University, Ithaca, NY 14853-4201}
\email{cao@math.cornell.edu}

\author{Bowei Liu}
\address{Department of Mathematics, %Building 380,
 Stanford University, CA 94305}
\email{bowei@math.stanford.edu}

\author{Ian Pendleton}
\address{Department of Mathematics,
 Cornell University, Ithaca, NY 14853-4201}
\email{iap26@cornell.edu}

\author{Abigail Ward}
\address{Department of Mathematics,
University of Chicago,
Chicago, IL 60637}
\email{abigailward@uchicago.edu}

\date{\today}

\begin{document}

\begin{abstract}
In this paper, we derive several differential Harnack estimates (also known as Li-Yau-Hamilton-type estimates) for positive solutions of Fisher's equation. We use the estimates to obtain lower bounds on the speed of traveling wave solutions and to construct classical Harnack inequalities. \end{abstract}

\maketitle

\section{Introduction}

Fisher's equation, or the Fisher-KPP partial differential equation, is given by
\vspace{12pt}
\begin{equation}
\label{eq:fisher}
f_t = \Delta f + c f ( 1 - f),
\end{equation}

where $f$ is a real-valued function on an n-dimensional Riemannian manifold $M^n$, and $c$ is a positive constant. 
The equation was proposed by R. A. Fisher in 1937 to describe the propagation of an evolutionarily advantageous gene in a population \cite{fisher37}, and was also independently described in a seminal paper by A. N. Kolmogorov, I. G. Petrovskii, and N. S. Piskunov in the same year \cite{kpp37}; for this reason, it is often referred to in the literature as the Fisher-KPP equation.
The density of the gene evolves according to diffusion (the term $\Delta f$) and reaction (the term $cf(1-f)$).
Since the two papers in 1937, the equation has found many applications including in the description of the branching Brownian motion process \cite{mckean75}, in neuropsychology \cite{tuckwell88}, and in describing certain chemical reactions \cite{naraigh13}.
Because a solution $f$ often describes a concentration or density, it is natural to study solutions to the equation for which $0 < f < 1$; our main theorems will simply assume positive solutions.

It is clear that $f=0$ and $f=1$ are stationary solutions to this equation on any manifold; it is also known that when $M^n=\R^n$ the equation admits traveling wave solutions, i.e. solutions $f(x,t) $ that we can express as a function of $z=x+\eta t$ for some vector $\eta \in \R^n.$ Under a broad range of conditions, general solutions to the equation in $\R^1$ approach a traveling wave solution with a unique minimal speed (see for example, \cite[Theorem 17]{kpp37} or \cite{fisher37, sherratt98}). A bound on the minimum speed of such a traveling wave solution on $\R^1$ was known to Kolmogorov-Petrovskii-Piskunov \cite{kpp37}; our work results in bounds for the minimum speed of a solution on $\R^n$ for $n=1,2,3$. While our bound in dimension $1$ is weaker than the previously known bounds, the bounds in higher dimensions are new and suggest that the study of Harnack inequalities may be used to bound the minimal speed of traveling waves in higher dimensions as well.
\vspace{12pt}

Our work introduces and proves three Li-Yau-Hamilton-type Harnack inequalities which constrain positive functions satisfying the Fisher-KPP equation on an arbitrary Riemannian manifold $M^n$. Depending on the setting we obtain different inequalities. The study of differential Harnack inequalities was first initiated by P. Li and S.-T. Yau in \cite{ly86} (also see \cite{ab79}). Harnack inequalities have since played an important role in the study of geometric analysis and geometric flows (for example, see \cite{hamilton93,  perelman1}). Applications have also been found to the study of nonlinear parabolic equations, e.g. in \cite{hamilton11}. One of these is a recent reproof of the classical result of H. Fujita \cite{fujita66}, which states that any positive solution to the Endangered Species Equation in dimension $n$,
\[f_t=\Delta f + f^p,\]
blows up in finite time provided $0<n(p-1)<2$; see \cite{cck14}.

When the dimension falls into a certain range we can integrate our differential Harnack inequality along any space-time curve to obtain a classical Harnack inequality which allows us to compare the values of positive solutions at any two points in space-time when time is large.

The organization for the paper is as follows: In Section 2 we present the precise formulations and the proofs of our two inequalities governing closed manifolds. In Section 3 we state and prove a similar Harnack inequality for complete noncompact manifolds. In Section 4, we end the paper with the aforementioned results on the minimum speed of traveling wave solutions and classical Harnack inequalities.
\vspace{12pt}

\paragraph{\textbf{Acknowledgements:}} X. Cao's research was
partially supported by NSF grant DMS 0904432. B. Liu, I. Pendleton, and A. Ward's research was supported by NSF grant DMS 1156350 through the Research Experience for Undergraduates Program at Cornell University during the summer of 2012. The authors would like to thank Professor Robert Strichartz for his encouragement and Benjamin Fayyazuddin-Ljungberg and Hung Tran for helpful discussions. They also thank the referee for many detailed suggestions which improved the quality of this paper.

\section{On Closed Manifolds}

In this section, we will deal with the case when the Riemannian manifold $M$ is closed, and we also assume that its Ricci curvature is non-negative. 

In what follows, the time derivative will always be taken to mean the derivative from the left, if the two-sided derivative does not exist.

    \subsection{Statement of Theorem} We first state our main theorem of this section. 
    
    \begin{theorem}
	\label{theorem:black}
    		Let  $(M^n,g)$ be an $n$-dimensional closed Riemannian manifold with non-negative Ricci curvature and $f(x,t) : M \times [0, \infty) \to \R$ be a positive solution of the Fisher-KPP equation $f_t = \Delta f + c f (1 - f)$, where $f$ is $C^2$ in $x$ and $C^1$ in $t$, and $c > 0$.
			
			Let $u= \log f$, then we have
			\begin{equation}
    			\Delta u + \alpha |\nabla u|^2 + \beta e^u + \phi_0(t) \geq 0
			\end{equation}
			for all $x$ and $t$, provided that
				\begin{itemize}
					\item[(i)] $0 < \alpha < 1$,
					\item[(ii)]$ \beta \leq \dfrac{-c n (1 + \alpha)}{4 \alpha ^2 - 4 \alpha + 2 n} < 0$,
					\item[(iii)] $ \dfrac{  8\beta ( 1 -\alpha)} n + c < 0$,
			\end{itemize}
			where
			\begin{equation*}
			\phi_0(t) = \frac{\left( \frac{\beta c n}{c n + 8 \beta (1 - \alpha)} \right) e^{-c t} - \beta}{1 - e^{-c t}}.
			\end{equation*}

  If instead of (iii) we have that
\begin{itemize}
\item[(iv)] $\dfrac{ 8 \beta  (1 - \alpha)} n + c \geq 0, $
\end{itemize}
then the following holds:
\begin{equation}
\Delta u + \alpha |\nabla u|^2 + \beta e^u + \phi_0(t) \geq 0,
\end{equation}
where
\begin{equation*}
\phi_0(t) = \begin{cases} \dfrac n { 2  (1- \alpha) t }, & t \leq T_2 := \dfrac n { 2 ( 1 - \alpha) ( - \beta c ) } \left ( \dfrac{ 4 \beta ( 1- \alpha)}n + c \right), \\ \dfrac{ - \beta c ( e^{c (t-T_2)} + 1 )}{ c+ \frac{ 8 \beta ( 1- \alpha)} n + c e^{ c (t-T_2)} }, & \text{otherwise}. \end{cases}  
\end{equation*}
		\end{theorem}
	
\begin{remark}
In summary, our theorem is that $\Delta u + \alpha |\nabla u|^2 + \beta e^u + \phi_0(t) \geq 0$, where
	\begin{equation*}
	\phi_0(t) = 
	\begin{cases}
	\dfrac{\left( \frac{\beta c n}{c n + 8 \beta (1 - \alpha)} \right) e^{-c t} - \beta}{1 - e^{-c t}},  &\text{if (iii)}, \\
	\dfrac n { 2  (1- \alpha) t }, & \text{if (iv) and }  t \leq  T_2, \phantom{\dfrac12} \\ 
	\dfrac{ - \beta c ( e^{c (t-T_2)} + 1 )}{ c+ \frac{ 8 \beta ( 1- \alpha)} n + c e^{ c (t-T_2)} }, & \text{if (iv) and } t > T_2.			\end{cases}
	\end{equation*}
\end{remark}	

\vspace{12pt}

\begin{remark}
We briefly describe the main idea of our proof here, which uses the parabolic maximum principle and an argument by contradiction. We first define a quantity $$h(x,t) : M \times (0,\infty) \to \R,$$ which will depend on a given solution to Fisher's Equation. We start with $h(x,\varepsilon)>0$ for any sufficiently small $\varepsilon > 0$, and our goal is to prove this quantity $h(x,t)$ remains positive for all points in $M \times \R^{+}$. As suggested in \cite{c08, ch09}, we then compute what we call the time evolution of $h$, namely $\frac{\partial h}{\partial t}$, in the following form:
\[ \frac{\partial h}{\partial t}(x,t) = \Delta h(x,t) + A_1(x,t) \cdot \nabla h(x,t) + A_2(x,t),\]

for some $A_1: M \times (0,\infty) \to \R^n$, and $A_2 : M \times
(0,\infty) \to \R$. We then assume for the sake of a contradiction that there exists a first (with respect to $t$) point $(x_1,t_1)$ where $h(x,t) \leq 0$; it follows that $\frac{\partial h}{\partial t}(x_1, t_1) \leq 0$. Since $h(x_1,t_1)$ must be a local minimum in $M$ of the function $h(x,t_1): M \to \R$, it also follows that $\Delta h(x_1,t_1) \geq 0$, and $\nabla h(x_1,t_1) = (0, \ldots, 0)$. Thus our time evolution simplifies to 
\[ \frac{\partial h}{\partial t}(x_1,t_1) \geq A_2(x_1,t_1) .\] 
By our construction of $h(x,t)$ we will force $A_2(x_1,t_1) > 0$, and so we will have
\[ 0 \geq \frac{\partial h}{\partial t} (x_1,t_1) \geq A_2 (x_1,t_1) > 0,\]
which is a contradiction. Thereby we conclude that $h(x,t) > 0$ for all $(x,t) \in M \times (0,\infty)$.
\end{remark}

 \subsection{Technical Lemmas}
 
In this section we prove the technical lemmas needed in the case that $M$ is a closed manifold. 

Lemma \ref{lemma:red} gives us the time evolution of $h$ in terms of 4 quantities $P_1, P_2, P_3, P_4$ (which sum to $A_2$ above). Lemma \ref{lemma:blue} gives a lower bound for $P_2$ which also applies in the noncompact case. Lemma \ref{lemma:green} introduces quantities $P_5, P_{5.1}, P_{5.2}$ which depend only on $\phi$ and which give a lower bound for $P_3$. Lemma \ref{lemma:yellow} puts a lower bound on $P_5$. Lemma \ref{lemma:purple}, used for our second Harnack inequality, bounds $P_3$ when Lemma \ref{lemma:yellow} is inapplicable. Finally, $P_1$ and $P_4$ are bounded in the proof of the main theorem.

\begin{lemma}
\label{lemma:red}
Let $(M^n, g)$ be a complete Riemannian manifold with Ricci curvature bounded from below by $\op{Ric} \geq -K$. Let $f(x,t): M^n \rightarrow \mathbb R$ be a positive solution to $f_t = \Delta f + c f (1 - f)$ which is $C^2$ in $x$ and $C^1$ in $t$. Let $u(x,t) = \log f(x,t)$, and let $\alpha, \beta, c$ be any constants. Define $h(x,t)$  as follows:
\begin{align*}
h(x,t) := \Delta u + \alpha |\nabla u|^2 + \beta e^u + \varphi, \\
\varphi = \varphi(x,t) = \phi(t) + \psi(x),
\end{align*}
where $\phi(t)$ is any $C^1$ function and $\psi(x)$ is any $C^2$ function. Then the following inequality holds:
\[
h_t - \Delta h - 2 \nabla u \cdot \nabla h \geq P_1 h + P_2 + P_3 + P_4, 
\]
where
\begin{align*}
P_1 & = \dfrac{2 ( 1 - \alpha)} n h - \dfrac{ 4 ( 1-  \alpha)} n (\alpha |\nabla u|^2 + \beta e^u + \phi + \psi ) -  c e^u , \\
\begin{split} 
P_2 & = \dfrac { 2 ( 1 - \alpha)} n ( \alpha^2 |\nabla u|^4 + 2 \phi \psi) - 2 K (1 - \alpha) |\nabla u|^2 + \dfrac{ 4  \alpha ( 1- \alpha)} n \phi |\nabla u|^2 \\
& \qquad \qquad +  |\nabla u|^2 e^u \left ( \dfrac { 4 \alpha \beta ( 1 - \alpha)}n - 2 \beta - \alpha c - c \right ) ,
\end{split} \\
P_3 & = e^{2u} \dfrac{ 2 \beta^2 ( 1 - \alpha)} n + e^u \left ( \dfrac{ 4 \beta ( 1 - \alpha )} n \phi + c \phi + c \beta \right ) + \dfrac{ 2 ( 1 - \alpha)} n \phi^2 + \phi_t , \\
P_4 & = \dfrac{ 4 \alpha ( 1-  \alpha)} n \psi |\nabla u|^2 - 2 \nabla u \cdot \nabla \psi + e^u \psi \left ( c + \dfrac{ 4 \beta ( 1 - \alpha)} n \right ) + \dfrac{ 2 ( 1-  \alpha)}n \psi^2 - \Delta \psi.
\end{align*}
\end{lemma}

\begin{remark}
Note that Lemma \ref{lemma:red} will be used in the proofs of both Theorem \ref{theorem:black} and Theorem \ref{theorem:white}, with different choices of $\alpha, \beta, c, \phi$ and $\psi$. The statement of Lemma \ref{lemma:red} is independent of these choices.
\end{remark}

\begin{proof}
The proof is based on a straightforward but long calculation. Let	$f : M \times [0,\infty) \to \R$ satisfy \eqref{eq:fisher}, hence $u$ must satisfy
	 \begin{equation*}
	   u_t = \Delta u + |\nabla u|^2 + c (1 - e^u).
  \end{equation*}

We then compute:
		   	\begin{align*}
				&(\partial_t - \Delta) u  = c - c e^u + |\nabla u|^2, \\
				&(\partial_t - \Delta) (\Delta u)  = \Delta |\nabla u|^2 - c (\Delta u) e^u - c |\nabla u|^2 e^u, \\
				&(\partial_t - \Delta) (\alpha |\nabla u|^2)  = 2 \alpha \nabla u \cdot \nabla (\Delta u) + 2 \alpha \nabla u \cdot \nabla |\nabla u|^2 - 2 \alpha c |\nabla u|^2 e^u - \alpha \Delta |\nabla u|^2, \\
				&(\partial_t - \Delta) (\beta e^u)  = \beta c e^u - \beta c e^{2u}, \\
				&(\partial_t - \Delta) \varphi(t) = \phi_t - \Delta \psi, \\
				&2 \nabla u \cdot \nabla h  = 2 \nabla u \cdot \nabla (\Delta u) + 2 \alpha \nabla u \cdot \nabla |\nabla u|^2 + 2 \beta |\nabla u|^2 e^u+2\nabla u \cdot \nabla \psi.
			\end{align*}
		Here we use the Weitzenb\"ock-Bochner formula,
			\begin{equation*}
				\Delta |\nabla u|^2 = 2 |\nabla \nabla u|^2 + 2 \nabla u \cdot \nabla (\Delta u) + 2 \op{Ric} (\nabla u, \nabla u),
			\end{equation*}
		where $\nabla \nabla u$ is the  Hessian  of $u(x,t)$.
		
		This leads to the equality
			\begin{align*}
				(\partial_t - \Delta) h - 2 \nabla u \cdot \nabla h & = 	
					2 (1 - \alpha) |\nabla \nabla u|^2 - c e^u (\Delta u) \\
					& \notag \quad - |\nabla u|^2 e^u (2 \alpha c + 2 \beta + c) + 2 (1 - \alpha) \op{Ric}(\nabla u, \nabla u) \\
					& \notag \quad + \beta c e^u - \beta c e^{2 u} + \phi_t- \Delta \psi - 2 \nabla u \cdot \nabla \psi. 
			\end{align*}
		Using Cauchy-Schwarz $|\nabla \nabla u|^2 \geq \dfrac 1 n (\Delta u)^2$ and $\op{Ric}\geq -K$ yields that  
			\begin{align*}
				(\partial_t - \Delta) h - 2 \nabla u \cdot \nabla h & \geq 	
					2 \frac{(1 - \alpha)}{n}(\Delta u)^2 - c e^u (\Delta u) \\
					& \notag \quad - |\nabla u|^2 e^u (2 \alpha c + 2 \beta + c) - 2 (1 - \alpha)K |\nabla u|^2 \\
					& \notag \quad + \beta c e^u - \beta c e^{2 u} + \phi_t- \Delta \psi - 2 \nabla u \cdot \nabla \psi.
			\end{align*}
		Finally, we substitute for $\Delta u$:
			\begin{equation*}
				\Delta u = h - \alpha |\nabla u|^2 - \beta e^u - \phi-\psi,
			\end{equation*}
		to expand and conclude that

\begin{align*}
    \begin{split}
        & h_t - \Delta h -  2 \nabla u \cdot \nabla h \geq \\
			      & \quad h \left ( \dfrac{ 2 ( 1-  \alpha)} n h - \dfrac{ 4 ( 1 - \alpha)} n ( \alpha |\nabla u|^2 + \beta e^u + \phi  + \psi ) - c e^u \right ) \\
        & \quad + \left [ \dfrac{ 2 ( 1 - \alpha)} n ( \alpha^2 | \nabla u |^4 + 2 \phi \psi ) - 2 K  ( 1 - \alpha) |\nabla u|^2 + \dfrac{4 \alpha (  1- \alpha)} n \phi |\nabla u|^2 \right. \\
				     & \qquad + \left. |\nabla u|^2 e^u \left ( \dfrac{ 4 \alpha \beta ( 1 - \alpha)} n - 2 \beta - \alpha c - c \right ) \right ] \\
        & \quad + \left [ e^{2u} \left ( \dfrac{ 2 \beta^2 ( 1 -\alpha)} n \right ) + e^u \left ( \dfrac{ 4 \beta ( 1- \alpha)} n \phi + c \phi + c \beta \right ) + \dfrac {2  (1 - \alpha)} n \phi^2 + \phi_t \right ] \\
        & \quad + \left [ \dfrac{ 4 \alpha ( 1 - \alpha)} n \psi |\nabla u|^2 - 2 \nabla u \cdot \nabla \psi + e^u \psi \left ( c + \dfrac{ 4 \beta ( 1 - \alpha)} n \right ) \right. \\
        & \qquad \left. + \dfrac{ 2 ( 1 - \alpha)} n \psi^2 - \Delta \psi \right]
    \end{split} \\
	 & \qquad \qquad \qquad = P_1 h + P_2 + P_3 + P_4,
\end{align*}
as desired.
\end{proof}

We now show that $P_2$ is nonnegative under the assumptions of Theorem \ref{theorem:black}.

\begin{lemma}
\label{lemma:blue}
If $K = 0$ and assuming that $(i), (ii)$ hold, then for any $x, t$ where $\phi(t), \psi(x) \geq 0$ we have that
\[ P_2 \geq 0. \]
\end{lemma}
\begin{proof}
We have assumed that $\alpha, 1 - \alpha, \phi, \psi, K \geq 0$. Note that
\[\dfrac{ 4 \alpha \beta ( 1 - \alpha)} n -  2 \beta - \alpha c - c \geq 0 \]
is equivalent to
\[( 4 \alpha ( 1 - \alpha) - 2 n)\beta - c n (\alpha + 1) \geq 0, \]
or
\[( -4 \alpha ( 1 - \alpha  ) + 2n ) \beta \leq - cn  (1  + \alpha), \]
which is exactly the condition $(ii)$ since $2n \geq 1 \geq 4 \alpha ( 1 - \alpha)$.
\end{proof}

Next, we find quantities depending only on $\phi$ which we will eventually use to guarantee that $P_3$ is strictly positive.

\begin{lemma}
\label{lemma:green}
Assume $\alpha < 1$.
Define
\begin{align*}
\mu_1 & := \dfrac c 2  \sqrt{ \dfrac n { 2 (1-\alpha)} }, \\
\nu_1 &= \dfrac{ c + \frac{ 4 \beta ( 1 - \alpha)}n }{ 2 \beta} \sqrt{ \dfrac{ n}{ 2 ( 1- \alpha)} } , \\
\omega_1 &= \sqrt { \dfrac{ 2 ( 1- \alpha)} n }, \\
P_5(\phi) & := - ( \mu_1 + \nu_1 \phi )^2 + ( \omega_1 \phi)^2 + \phi_t, \\
P_{5.1}(\phi) & := \left ( \dfrac {4 \beta ( 1- \alpha)} n + c \right ) \phi + \beta c, \\
P_{5.2}(\phi) & := \dfrac { 2  (1- \alpha)}n \phi^2 + \phi_t.
\end{align*}

Then for any $(x,t)$, $P_5 > 0$ implies that $P_3 > 0$. Alternatively, if $P_{5.1} \geq 0$ and $P_{5.2} > 0$, then $P_3 > 0$.
\end{lemma}
\begin{proof}
Recall that
\[ P_3 (\phi) =   e^{2u} \left ( \dfrac{ 2 \beta^2 ( 1 -\alpha)} n \right ) + e^u \left ( \dfrac{ 4 \beta ( 1- \alpha)} n \phi + c \phi + c \beta \right ) + \dfrac {2  (1 - \alpha)} n \phi^2 + \phi_t.
\]
If $P_5 > 0$, then by using that $x^2 + 2 xy \geq - y^2$, where $x^2 = e^{2u}\left( \frac{2 \beta^2 (1-a)}{n} \right)$, we get that
\begin{align*}
	P_3(\phi) 	& \geq  - \dfrac{n}{8(1-\alpha) \beta^2} \left [ \beta c +  \left ( c  + \dfrac{4 ( 1- \alpha) \beta }n \right ) \phi \right ]^2 + \dfrac{ 2 ( 1 - \alpha)} n \phi^2 + \phi_t \\
		& = - ( \mu_1 + \nu_1 \phi ) ^2 + (\omega_1 \phi)^2 + \phi_t = P_5(\phi) > 0.
	\end{align*}
Alternatively, if $P_{5.1} \geq 0$ and $P_{5.2} > 0$, then since $(1-\alpha) > 0$ we can ignore the first term of $P_3$ and get
	\begin{align*}
		P_3 (\phi) & \geq   e^u \left ( \dfrac{ 4 \beta ( 1- \alpha)} n \phi + c \phi + c \beta \right ) + \dfrac {2  (1 - \alpha)} n \phi^2 + \phi_t  \\
& = e^u P_{5.1} + P_{5.2} > 0.
\end{align*}
\end{proof}

We now find functions $\phi(t)$ such that $P_3(\phi) > 0$. In Lemma \ref{lemma:yellow} we construct $\phi(t)$ in the case that (iii) is true, and in Lemma \ref{lemma:purple} we construct $\phi(t)$ when (iv) is true.

\begin{lemma}
\label{lemma:yellow}
Let $\mu, \nu, \omega$ be any constants such that $\mu \neq 0$, $\nu^2 < \omega^2$ and $\omega > 0$. If for sufficiently small $\eps > 0$ we define
\[
\phi(t) := \dfrac{ \mu \left ( \dfrac 1 { \nu - ( \omega - \eps )} e^{2 \mu ( \omega - \eps) t} - \dfrac 1 { \nu + (\omega - \eps ) } \right ) } { 1  - e^{2 \mu ( \omega - \eps ) t } },
\]
then 
\begin{equation*}
- ( \mu + \nu \phi)^2 +  (\omega \phi)^2  + \phi_t > 0,
\end{equation*}
where $\dis \lim_{t \rightarrow 0^+} \phi(t) = \infty$ and $\phi(t) \geq 0$ for all $t$.
\end{lemma}
\begin{proof}
    Choose $\eps$ small enough so that $\nu^2 < ( \omega - \eps )^2$. We claim that $\phi(t) $ satisfies the following equation:
    \begin{equation*}
      - ( \mu + \nu \phi)^2 +  [(\omega-\eps) \phi]^2  + \phi_t(t) = 0
    \end{equation*}
    for all time. This follows from the direct computation below. On the one hand we get that
\begin{align*}
    - ( \mu + \nu \phi)^2 +  [(\omega-\eps) \phi]^2	& =  \frac{\mu^2 (\omega - \eps)^2 \left( \frac{e^{2 \mu (\omega-\eps) t}}{\nu - (\omega - \eps)} - \frac{1}{\nu + (\omega- \eps)}\right)^2}{\left( 1 - e^{2 \mu (\omega - \eps) t} \right)^2} \\
    & \quad - \left( \mu + \frac{\mu \nu \left( \frac{e^{2\mu(\omega - \eps) t}}{\nu - (\omega - \eps)} - \frac{1}{\nu + (\omega-\eps)} \right)}{1- e^{2 \mu (\omega - \eps) t}}  \right)^2 \\
     & = \dfrac{ \mu^2[ 2 (\omega - \eps)( \omega - \eps - \nu ) ]\left[ 2 (\omega - \eps)( \omega - \eps + \nu ) e^{2 \mu ( \omega - \eps) t }\right] }{(1- e^{2 \mu (\omega - \eps)t})^2 (\nu - (\omega - \eps))^2(\nu + (\omega - \eps))^2} \\
     & = - \dfrac{4 \mu^2 (\omega - \varepsilon)^2 e^{2 \mu (\omega - \varepsilon) t}}{(\nu+(\omega - \varepsilon))(\nu-(\omega - \varepsilon))(e^{2 \mu(\omega - \varphi)t} - 1)^2}.
    \end{align*}
    On the other hand we have
    \begin{gather*}
    	\phi_t(t)		= \frac{2 \mu^2 (\omega - \eps) e^{2 \mu (\omega - \eps)t} }{(1-e^{2 \mu (\omega - \eps)t})(\nu - (\omega - \eps))} + \frac{2 \mu^2 (\omega- \eps) e^{2 \mu (\omega - \eps)t} \left(\frac{e^{2 \mu (\omega - \eps)t}}{\nu - (\omega-\eps)} - \frac{1}{\nu + (\omega - \eps)} \right)}{(1 - e^{2 \mu (\omega - \eps)t})^2} \\
    		 =  \frac{4 \mu^2 (\omega - \eps)^2 e^{2\mu (\omega - \eps) t}}{ (\nu + (\omega - \eps))(\nu - (\omega - \eps))(1- e^{2 \mu (\omega - \eps)t})^2}.
    \end{gather*}
    Therefore it follows that
    \begin{equation*}
      - ( \mu + \nu \phi)^2 +  [(\omega-\eps) \phi]^2  + \phi_t = 0,
    \end{equation*}
    and hence
    \begin{equation*}
    	- (\mu + \nu \phi)^2 + (\omega \phi)^2 + \phi_t = 2 \eps \omega \phi^2 - \eps ^2 \phi^2 = \phi^2 (2 \eps \omega - \eps ^2).
    \end{equation*}
    
    Note that $\nu - (\omega - \eps)$ and $\nu + (\omega- \eps)$ must have different signs since their product is $\nu^2  - (\omega - \eps)^2 < 0$; hence $\phi(t) \neq 0$ for all time. It  then follows that for sufficiently small $\eps$,
\begin{equation*}
- ( \mu + \nu \phi)^2 + (\omega \phi)^2 + \phi_t = \phi^2(2 \varepsilon \omega - \varepsilon^2) > 0.
\end{equation*}

To show that $\lim_{t \to 0^+} \phi(t) = \infty$, we split $\phi(t)$ into two parts. First, note that
\begin{align*}
	\lim_{t \to 0^+} \left( \frac{1}{\nu-(\omega - \varepsilon)} e^{2 \mu (\omega - \varepsilon)t} - \frac{1}{\nu + (\omega - \varepsilon)} \right) & = \frac{1}{\nu-(\omega - \varepsilon)} - \frac{1}{\nu + (\omega - \varepsilon)} \\
		& = \frac{2(\omega - \varepsilon)}{\nu^2-(\omega - \varepsilon)^2} < 0.
\end{align*}
Further, it is clear that
\begin{equation*}
	\lim_{t \to 0^+} \frac{\mu}{1-e^{2 \mu (\omega - \varepsilon) t}} = - \infty.
\end{equation*}
 Combining these two calculations lets us conclude that
 \begin{align*}
 	\lim_{t \to 0^+} \phi(t) = \infty.
 \end{align*}
Finally, since $\phi(t)$ is continuous and starts out positive and $\phi(t) \neq 0$ for any $t > 0$, it follows that $\phi(t) > 0$ for all $t > 0$.    
\end{proof}
    
\begin{remark}
We can also compute $\dis \lim_{t \rightarrow \infty} \phi(t)$.

If $\mu  >0$ then $e^{2 \mu ( \omega - \eps)t } \to \infty$ as $ t \rightarrow \infty$; hence we find that
\begin{equation*}
\lim_{t \rightarrow \infty} \phi(t) = \dfrac{ \frac{ \mu }{ \nu - (\omega - \eps)} }{ -1}  = \dfrac {\mu }{-\nu+(\omega - \eps)}.
\end{equation*}
If $\mu < 0$ then $e^{2 \mu (\omega - \eps)t} \rightarrow 0 $ as $ t \rightarrow \infty$, which gives us
\begin{equation*}
\lim_{t \rightarrow \infty} \phi(t) = \dfrac{ - \mu } {\nu + (\omega - \eps)}.
\end{equation*}

\end{remark}

Next we deal with the other case.

\begin{lemma}
\label{lemma:purple}
Let $\mu_1, \nu_1, \omega_1$ be defined as in Lemma \ref{lemma:green}, and suppose (iv) is true (i.e., (iii) becomes false). Let 
\begin{equation*}
T_2 = T_2(\eps) := \dfrac n { 2 ( 1- \alpha) ( 1-  \eps ) ( - \beta c )} \cdot
\left ( \dfrac { 4 \beta ( 1- \alpha)} n  + c \right ). 
\end{equation*}

If for some sufficiently small $\eps > 0$ we define
\begin{equation*} \phi(t) := \begin{cases}
\dfrac n { 2 ( 1 - \alpha)( 1 -\eps) t }, & t \leq T_2; \\
\dfrac{ - \mu_1 ( e^{ 2 \mu_1 (\omega_1 - \eps) 	(t - T_2)} + 1 ) }
{( \nu_1 + (\omega_1 -\eps) ) + ( \nu_1 - ( \omega_1 - \eps)) e^{2 \mu_1 (\omega_1 - \eps) ( t- T_2)} }, & t > T_2,
\end{cases} 
\end{equation*}
then for $t \leq T_2$ we get $P_{5.1} \geq 0$ and $P_{5.2}> 0$, and for $t > T_2$ we get $P_5 > 0$. Therefore $P_3(\phi) > 0$ for all $t$. 

In addition, $\lim_{t \rightarrow 0^+} \phi(t) = \infty$ and $\phi(t) > 0$ for all $t$.
\end{lemma}

\begin{proof}

We have that for $\varepsilon < 1$, 
\[ \lim_{t \to 0^+} \phi(t) = \lim_{t \to 0^+} \frac{n}{2(1-\alpha)(1-\varepsilon)t} = \infty.\]

To show that $\phi(t)$ is continuous at $T_2$, we check its limit from the left and from the right.

The limit from the left:
	\begin{align*}
	\lim_{t \to T_2^-} \phi(t) & = \frac{n}{2(1-\alpha)(1-\varepsilon)T_2} \\
		& = \frac{-\beta c n}{4 \beta (1-\alpha) + cn}.
	\end{align*}
	Now the limit from the right:
	\begin{align*}
		\lim_{t \to T_2^+} \phi(t) & = \frac{-\mu_1(1+1)}{(\nu_1+(\omega_1 - \varepsilon)) + (\nu_1-(\omega_1 - \varepsilon))} \\
			& = \frac{-2\mu_1}{2\nu_1}\\
			& = - \frac{c}{2} \cdot \frac{2 \beta n}{(cn + 4 \beta(1-\alpha))} \\
			& = \frac{-\beta c n}{4 \beta (1-\alpha) + cn}.
	\end{align*}
	Therefore $\phi(t)$ is continuous.
	
	Next we check that $\phi(t) > 0$ for all $t > 0$. Note that $\phi(t)$ is continuous, and clearly is positive between 0 and $T_2$. For $t \geq T_2$, since $\mu_1 \neq 0$, it follows that 
	\begin{equation*}
		- \mu_1 ( e^{ 2 \mu_1 (\omega_1 - \eps) 	(t - T_2)} + 1 ) \neq 0,
	\end{equation*}
	and therefore $\phi(t) \neq 0$ for any $t \geq T_2$. By continuity, it follows that $\phi(t) > 0$ for all $t > 0$.
	
	\vspace{12pt}
	Next we show that for $t \leq T_2$ we have that $P_{5.1} \geq 0$. That is, we need
	\begin{equation*}
	P_{5.1}=\left( \frac{4 \beta (1-\alpha)}{n} + c \right)\phi(t) + \beta c \geq 0.
	\end{equation*}
	First we note that condition (iv) states that $\left( \frac{ 4 \beta (1-\alpha)}{n} + c \right) \geq 0$. Since $\phi(t)$ is decreasing in $t < T_2$ it suffices to check that $P_{5.1}\geq 0$ holds for $t=T_2$:
	
	\begin{align*}
		\left( \frac{4 \beta (1-\alpha)}{n} + c \right) \phi(t) + \beta c & \geq 	\left( \frac{4 \beta (1-\alpha)}{n} + c \right) \phi(T_2) + \beta c \\
		& = \left( \frac{4 \beta (1-\alpha)}{n} + c \right)\left(\frac{-\beta c}{\left( \frac{4 \beta (1-\alpha)}{n} + c \right)} \right) + \beta c \\
		& = 0.
	\end{align*}
	Therefore $P_{5.1}\geq 0$ for all $t \leq T_2$.
	
	\vspace{12pt}
	Now we show that $P_{5.2}>0$ for all $t \leq T_2$. That is, we need
	\begin{equation*}
		P_{5.2}= \frac{2(1-\alpha)}{n} \phi(t)^2 + \phi_t(t) > 0.
	\end{equation*}
	We have that
	\begin{align*}
		P_{5.2} & = \frac{2(1-\alpha)}{n} \left[ \frac{n}{2(1-\alpha)(1-\varepsilon)t} \right]^2 + \frac{-n}{2(1-\alpha)(1-\varepsilon)t^2} \\
		& = \frac{n}{2(1-\alpha)(1-\varepsilon)^2 t^2} - \frac{n}{2(1-\alpha)(1-\varepsilon)t^2} \\
		& = \frac{\varepsilon n}{2 (1-\alpha)(1-\varepsilon)^2 t^2} >0 .
	\end{align*}
	This implies that $P_3(\phi)>0$ for $t\leq T_2$. Next we show that $P_5>0$ for all $t > T_2$. That is, we need that
	\begin{equation*}
		P_5 = -(\mu_1 + \nu_1 \phi)^2 + (\omega_1 \phi)^2 + \phi_t> 0
	\end{equation*}
	for
	\begin{equation*}
		\phi(t) = \dfrac{ - \mu_1 ( e^{ 2 \mu_1 (\omega_1 - \eps) 	(t - T_2)} + 1 ) }
{( \nu_1 + (\omega_1 -\eps) ) + ( \nu_1 - ( \omega_1 - \eps)) e^{2 \mu_1 (\omega_1 - \eps) ( t- T_2)} }.
	\end{equation*}
We first show that for $t > T_2$, $\phi(t)$ satisfies
\begin{equation*} 
- (\mu_1 + \nu_1 \phi)^2 + [(\omega_1 - \eps)\phi]^2 + \phi_t = 0.
\end{equation*}
	Plugging in $\phi(t)$ for $t > T_2$ gives us that
	\begin{align*}
	 -(\mu_1+\nu_1 \phi)^2 + [(\omega_1 & - \varepsilon) \phi ]^2 = \\
		& = -\left[ \mu_1 - \dfrac{ \mu_1 \nu_1 ( e^{ 2 \mu_1 (\omega_1 - \eps) 	(t - T_2)} + 1 ) } 
{( \nu_1 + (\omega_1 -\eps) ) + ( \nu_1 - ( \omega_1 - \eps)) e^{2 \mu_1 (\omega_1 - \eps) ( t- T_2)} }\right]^2\\
		& \hspace{12pt} + \left[ (\omega_1 - \varepsilon) \dfrac{ - \mu_1 ( e^{ 2 \mu_1 (\omega_1 - \eps) 	(t - T_2)} + 1 ) }
{( \nu_1 + (\omega_1 -\eps) ) + ( \nu_1 - ( \omega_1 - \eps)) e^{2 \mu_1 (\omega_1 - \eps) ( t- T_2)} } \right]^2 \\
			& = \dfrac{\mu_1^2 (\omega_1 - \varepsilon)^2 \left[  - \left(1 - e^{2 \mu_1(\omega_1 - \varepsilon)(t-T_2)}  \right)^2 + \left(e^{2 \mu_1(\omega_1 - \varepsilon)(t-T_2)} + 1 \right)^2 \right]}{\left[ (\nu_1 + (\omega_1 - \varepsilon)) + (\nu_1 - (\omega_1 - \varepsilon))e^{2 \mu_1 (\omega_1 - \varepsilon) (t - T_2)} \right]^2} \\
			& = \frac{4 \mu_1^2 (\omega_1 - \varepsilon)^2 e^{2\mu_1 (\omega_1 - \varepsilon) (t-T_2)}}{\left[ (\nu_1 + (\omega_1 - \varepsilon)) + (\nu_1 - (\omega_1 - \varepsilon))e^{2 \mu_1 (\omega_1 - \varepsilon) (t - T_2)} \right]^2}.
	\end{align*}
	
	Similarly, we have that
	\begin{align*}
		\phi_t(t) & = \dfrac{- 2\mu_1^2(\omega_1 - \varepsilon)e^{2\mu_1(\omega_1 - \varepsilon)(t-T_2)} \left[ (\nu_1 + (\omega_1 - \varepsilon)) + (\nu_1 - (\omega_1 - \varepsilon))e^{2 \mu_1 (\omega_1 - \varepsilon)(t-T_2)} \right]}{\left[ (\nu_1 + (\omega_1 - \varepsilon)) + (\nu_1 - (\omega_1 - \varepsilon))e^{2 \mu_1 (\omega_1 - \varepsilon) (t - T_2)} \right]^2}\\
			& \hspace{12pt} - \left[ \dfrac{ (\nu_1 - (\omega_1 - \varepsilon))(2 \mu_1 (\omega_1 - \varepsilon)) e^{2 \mu_1 (\omega_1 - \varepsilon)(t-T_2)} \left[- \mu_1 (e^{2 \mu_1 (\omega_1 - \varepsilon)(t-T_2)} +1)\right] }{\left[ (\nu_1 + (\omega_1 - \varepsilon)) + (\nu_1 - (\omega_1 - \varepsilon))e^{2 \mu_1 (\omega_1 - \varepsilon) (t - T_2)} \right]^2} \right]\\
			& = - \frac{4 \mu_1^2 (\omega_1 - \varepsilon)^2 e^{2\mu_1 (\omega_1 - \varepsilon) (t-T_2)}}{\left[ (\nu_1 + (\omega_1 - \varepsilon)) + (\nu_1 - (\omega_1 - \varepsilon))e^{2 \mu_1 (\omega_1 - \varepsilon) (t - T_2)} \right]^2}.
	\end{align*}
Therefore
\begin{equation*}
 - (\mu_1 + \nu_1 \phi)^2 + [(\omega_1 - \eps)\phi]^2 + \phi_t = 0,
\end{equation*}
and it follows that
\begin{equation*}
P_5 = - (\mu_1 + \nu_1 \phi)^2 + (\omega_1 \phi)^2 + \phi_t \\
 =  (2 \varepsilon \omega_1 - \varepsilon^2) \phi^2 > 0
\end{equation*}
for small enough $\varepsilon$. Therefore $P_3(\phi) > 0$ for $t > T_2$.
\end{proof}

\begin{remark}
Here we observe that
	\begin{align*}
		\lim_{t \to \infty} \phi(t) & = \lim_{t \to \infty} \dfrac{ - \mu_1 ( e^{ 2 \mu_1 (\omega_1 - \eps) 	(t - T_2)} + 1 ) }  {( \nu_1 + (\omega_1 -\eps) ) + ( \nu_1 - ( \omega_1 - \eps)) e^{2 \mu_1 (\omega_1 - \eps) ( t- T_2)} } \\
			& = \frac{-\mu_1}{\nu_1-(\omega_1-\varepsilon)} = \frac{\mu_1}{-\nu_1+(\omega_1-\varepsilon)},
\end{align*}
which is the same limit as $\phi(t)$ from Lemma \ref{lemma:yellow} since $\mu_1 > 0$.
\end{remark}

  Now we are ready to finish the proof of Theorem \ref{theorem:black}.

    \subsection{Proof of Theorem \ref{theorem:black}}

Let  $f : M \times [0, \infty) \to \R$ be a positive solution of $f_t = \Delta f + c f (1 - f)$ for $c > 0$, and assume that the following hold:
        \begin{itemize}
    				\item[(i)] $0 < \alpha < 1$,
					\item[(ii)]$ \beta \leq \frac{-c n (1 + \alpha)}{4 \alpha ^2 - 4 \alpha + 2 n} < 0$.
			\end{itemize}

        Let $u = \log f$, and define
            \begin{equation*}
                h(x,t)  := \Delta u + \alpha |\nabla u|^2 + \beta e^u + \varphi,
              \end{equation*}
              where
              \begin{equation*}
              \varphi = \varphi(x,t) = \phi(t) + \psi(x)
              \end{equation*}
              and since we are in the closed case we set $\psi(x)=0$.
                                              
              With $\mu_1, \nu_1, \omega_1$, and $T_2$ as defined in Lemma \ref{lemma:green} and Lemma \ref{lemma:purple}, and $\varepsilon > 0$ small enough to satisfy Lemmas \ref{lemma:yellow} and \ref{lemma:purple}, we let
                \begin{align*}
                \phi(t) & = \begin{cases}
 \dfrac{ \mu_1 \left ( \dfrac 1 { \nu_1 - ( \omega_1 - \eps )} e^{2 \mu_1 ( \omega_1 - \eps) t} - \dfrac 1 { \nu_1 + (\omega_1 - \eps ) } \right ) } { 1  - e^{2 \mu_1 ( \omega_1 - \eps ) t } }, &  \text{ if (iii)}, \\
\dfrac n { 2 ( 1 - \alpha)( 1 -\eps) t }, & \text{ if (iv) and } t \leq T_2, \\
\dfrac{ - \mu_1 ( e^{ 2 \mu_1 (\omega_1 - \eps) 	(t - T_2)} + 1 ) }
{( \nu_1 + (\omega_1 -\eps) ) + ( \nu_1 - ( \omega_1 - \eps)) e^{2 \mu_1 (\omega_1 - \eps) ( t- T_2)} }, & \text{ if (iv) and } t > T_2.
\end{cases}
            \end{align*}

We first show that $h(x,t)>0$ for all $t$. Suppose for the sake of a contradiction that $h \leq 0$ somewhere; let $t_1$ be the first time such that $\min_x h(x,t) = 0$. Since $M$ is closed the minimum is attained, say at the point $(x_1,t_1)$. By Lemmas \ref{lemma:yellow} and \ref{lemma:purple}, $\lim_{t \to 0^+} \phi(t)  =\infty$ so it follows that $t_1$ exists.

By applying Lemma \ref{lemma:red}, we get that
\begin{equation} \label{eq:contra}
h_t - \Delta h - 2 \nabla u \cdot \nabla h \geq P_1 h + P_2 + P_3 + P_4, 
\end{equation}
where $P_1, \ldots, P_4$ are defined as in Lemma \ref{lemma:red}. Note that in the case (iv), the derivative $\phi_t$ at $t = T_2$ is considered to be the derivative from the left.

We have $P_1 h = 0 $ since $h(x_1,t_1) = 0$. Lemma 2 yields that $P_2 \geq 0$ since $K = 0$, and $P_4 = 0$ since $\psi(x) \equiv 0$.

Since $(x_1,t_1)$ is the first space-time where $h(x,t) = 0$, the maximum principle yields that $h_t (x_1,t_1) \leq 0$ (where this is a derivative as $t \to t_1^-$), $\Delta h(x_1,t_1) \geq 0$ and $\nabla h (x_1,t_1) = 0$.

Hence \eqref{eq:contra} yields that 
\begin{equation}
\label{eq:doom}
 0 \geq h_t - \Delta h - 2 \nabla u \cdot \nabla h \geq P_1 h + P_2 + P_3 + P_4 \geq P_3.
\end{equation}

Now we split into cases based on whether (iii) or (iv) holds.

If (iii) is true, since $c> 0$ we have the following inequalities:
\begin{align*}
\dfrac{ 4 \beta ( 1 - \alpha)} n < c + \dfrac{ 4 \beta  (1 - \alpha)} n  < - \dfrac{ 4 \beta ( 1 - \alpha)} n, \\
\left | c + \dfrac {4  \beta  (1 - \alpha)} n \right | < \left | \dfrac{ 4 \beta ( 1- \alpha)}n \right | , \\
\left ( \dfrac{ c + \frac{4 \beta ( 1- \alpha)} n } { 2 \beta } \right )^2  < \left ( \dfrac{ 2  (1-  \alpha)} n \right )^2 ,\\
\nu_1^2 = \left ( \dfrac{ c + \frac{4 \beta ( 1- \alpha)} n } { 2 \beta } \right )^2 \dfrac{ n }{ 2 ( 1- \alpha)} < \omega_1^2 = \dfrac{ 2  (1 - \alpha)} n.
\end{align*}
Therefore by Lemmas \ref{lemma:green} and \ref{lemma:yellow} it follows that $P_3 > 0$, which contradicts \eqref{eq:doom}.

Otherwise, if (iv) is true it follows from Lemmas \ref{lemma:green} and \ref{lemma:purple} that $P_3 > 0$ again, which still contradicts \eqref{eq:doom}.

This proves that $h(x,t) > 0$ for all $x, t$. Finally, letting $\eps \to 0$ with
        \begin{equation*}
        \left. \phantom{\dfrac 12} T_2 \right |_{\eps = 0} = \dfrac n { 2 (  1 - \alpha )( - \beta c )} \left ( \dfrac{ 4 \beta ( 1 - \alpha)} n + c \right ),
       \end{equation*}
we get that $\phi(t) \to \phi_0(t)$ where
	\begin{equation*}\phi_0(t) = \begin{cases}
\dfrac{\left( \frac{\beta c n}{c n + 8 \beta (1 - \alpha)} \right) e^{-c t} - \beta}{1 - e^{-c t}},  &\text{if (iii)}, \\
\dfrac n { 2  (1- \alpha) t }, & \text{if (iv) and} \! \! \! \left. \phantom{\dfrac 12} t \leq  T_2  \right |_{\eps = 0},  \\ 
\dfrac{ - \beta c ( e^{c (t-T_2)} + 1 )}{ c+ \frac{ 8 \beta ( 1- \alpha)} n + c e^{ c (t-T_2)} }, & \text{if (iv) and} \! \! \! \left. \phantom{\dfrac12} t > T_2 \right |_{\eps = 0}. \end{cases}
	\end{equation*}
	Therefore
	\begin{equation*}
 \lim_{\eps \to 0} h(x,t) = \Delta u + \alpha |\nabla u|^2 + \beta e^u + \phi_0(t) \geq 0
        \end{equation*}
        as desired.
     
%%%%%%%%%%%%%%%%%%%%%%%%
\section{On Complete Noncompact Manifolds}%%%%
%%%%%%%%%%%%%%%%%%%%%%%%
In this section, we study the case in which the manifold is complete but noncompact. The idea is similar to the case when the manifold is compact without boundary. The main technical difficulty here is to ensure that the minimum of the Harnack quantity is attained in a compact region. We first state our main theorem of this section. 

	\subsection{Statement of Theorem }
		
		\begin{theorem}
		\label{theorem:white}
		Let  $(M^n,g)$ be an $n$-dimensional complete (noncompact) Riemannian manifold with non-negative Ricci curvature. Let $f(x,t):M \times [0, \infty) \to \R$ be a positive solution of the Fisher-KPP equation $f_t = \Delta f + c f(1 - f)$, where $f$ is $C^2$ in $x$ and $C^1$ in $t$, and $c>0$ is a constant. Let $u= \log f$, then we have
			\begin{equation}
	   			 \Delta u + \alpha |\nabla u|^2 + \beta e^u + \phi_1(t) \geq 0,
			\end{equation}
			provided the following constraints are satisfied:
				\begin{itemize}
	    			\item[(i)] $0 < \alpha < 1$,
	   				\item[(ii)]  $\beta < \dfrac{-c n(1 + \alpha)}{2 \left(2 \alpha^2 - 2 \alpha + n \right)} < 0$,
					\item[(iii)] $ \displaystyle \dfrac{ -cn ( 2 + \sqrt 2)}{ 4 ( 1- \alpha)} < \beta < \dfrac{ -cn ( 2 - \sqrt 2 )}{4 ( 1 - \alpha)}, $
				\end{itemize}
		where
		\begin{equation*}
		\phi_1(t) = \dfrac{ \mu_2 \left ( \frac 1 { \nu_2 - \omega_2} e^{ 2 \mu_2 \omega_2 t } - \frac 1 { \mu_2 + \omega_2 } \right )}{ 1 - e^{2 \mu_2 \omega_2 t}},
		\end{equation*}
		with
		\begin{align*}
		\mu_2 & =  \beta c \sqrt{ \frac{ 2 ( 1- \alpha)} {c ( - cn - 8 \beta ( 1- \alpha))} }, \\
		\nu_2 & = \left ( \frac{ 4 \beta ( 1- \alpha)} n + c \right ) \cdot \sqrt { \frac{ 2  ( 1- \alpha)} { c ( - cn  - 8 \beta ( 1- \alpha))} }, \\
		\omega_2 & =  \sqrt{  \frac { 2 ( 1 -\alpha)} n }.
		\end{align*}
		\end{theorem}

\vspace{12pt}
\subsection{Technical Lemmas}

In this subsection, we state and prove some additional lemmas which will be needed in the proof of Theorem 2. Lemma \ref{lemma:orange} allows us to substitute the sum $P_6 + P_7$ for  $P_3 + P_4$; then Lemma \ref{lemma:cyan} bounds $P_6$ using a new quantity $P_8$. Lemma \ref{lemma:magenta} allows us to apply Lemma \ref{lemma:yellow} to control $P_8$. Lemma \ref{lemma:brown} gives sufficient conditions for bounding $P_7$. After bounding $P_1$, we are in a position to prove our theorem.

For any given $\varepsilon'>0$, let 
\begin{equation*} 
A = A(\eps') := 
\dfrac{ 2 \beta^2 ( 1 - \alpha)} n - \dfrac{  n \left ( c + \frac{ 4 \beta  (1 - \alpha)} n \right )^2 }{8 ( 1 - \alpha - \eps' )}.
\end{equation*}

\begin{lemma}
\label{lemma:orange}
Let $P_3$ and $P_4$ be as defined in Lemma \ref{lemma:red}. Define
\begin{align*}
\begin{split}
P_6 & := A e^{2u}  + e^u \left ( \dfrac{ 4 \beta (1 - \alpha) \phi}n + c \beta + c \phi \right ) + \dfrac{ 2 (1 - \alpha)} n \phi^2 + \phi_t,
\end{split} \\
 P_7 & := \dfrac{ 4 \alpha  (1- \alpha)} n \psi |\nabla u|^2  - 2 \nabla u \cdot \nabla \psi + \dfrac{ 2 \eps'} n \psi^2 - \Delta \psi. 
\end{align*}
For any $\varepsilon'>0$ and any $(x,t)$ we have
\[ P_3 + P_4 \geq P_6 + P_7. \]
\end{lemma}

\begin{proof}[Proof of Lemma 6]
Recall that, 
\begin{equation*}
\begin{split}
P_3 + P_4  =&  \dfrac{ 2 \beta^2 ( 1 - \alpha)} n e^{2u} + e^u \left ( \dfrac{ 4 \beta ( 1- \alpha)} n \phi + c \phi + c \beta \right ) \\
& + \dfrac{ 2 ( 1- \alpha)} n \phi^2 + \phi_t
+ \dfrac{4 \alpha ( 1 - \alpha)}n \psi |\nabla u|^2 - 2 \nabla u \cdot \nabla \psi  \\
& - \Delta \psi+ e^u \psi \left ( c + \dfrac{ 4 \beta ( 1 - \alpha)} n \right ) + \dfrac{ 2 ( 1-  \alpha)} n \psi^2 .
\end{split}
\end{equation*}
We write the last two terms as
\begin{align*}
& e^u \psi \left ( c + \dfrac{ 4 \beta ( 1 - \alpha)} n \right ) + \dfrac{ 2 ( 1-  \alpha)} n \psi^2 \\
& = e^u \psi \left ( c + \dfrac{ 4 \beta ( 1 - \alpha)} n \right ) + \dfrac{ 2 ( 1-  \alpha - \eps')} n \psi^2 + \dfrac{ 2 \eps'}n \psi^2.
\end{align*}

Using $2xy+ x^2 \geq - y^2$ in the form
$$e^u \psi \left ( c + \dfrac{ 4 \beta ( 1 - \alpha)} n \right )  + \dfrac{ 2 ( 1-  \alpha - \eps')} n \psi^2\geq -\dfrac{ n \left ( c + \frac{ 4 \beta ( 1-  \alpha)} n \right)^2 }{ 8 ( 1 - \alpha - \eps')} e^{2u},$$

gives us
\begin{align*}
e^u \psi \left ( c + \dfrac{ 4 \beta ( 1 - \alpha)} n \right ) + \dfrac{ 2 ( 1-  \alpha)} n \psi^2 
& \geq \dfrac{ 2 \eps'} n \psi^2 - \dfrac{ n \left ( c + \frac{ 4 \beta ( 1-  \alpha)} n \right)^2 }{ 8 ( 1 - \alpha - \eps')} e^{2u}.
\end{align*}
Applying this inequality then gives
\begin{align*}
\begin{split}
P_3 + P_4 & \geq e^{2u} \left ( \dfrac{ 2 \beta^2 ( 1 - \alpha)} n - \dfrac{  n \left ( c + \frac{ 4 \beta  (1 - \alpha)} n \right )^2 }{8 ( 1 - \alpha - \eps' )} \right ) \\
& \qquad + e^u \left ( \dfrac{  4 \beta (1 - \alpha) \phi}n + c \beta + c \phi \right ) + \dfrac{ 2 (1 - \alpha)} n \phi^2 + \phi_t \\
& \qquad + \dfrac{ 4 \alpha  (1- \alpha)} n \psi | \nabla u|^2 - 2 \nabla u \cdot \nabla \psi + \dfrac{ 2 \eps'} n \psi^2 - \Delta \psi, 
\end{split} \\
& = P_6 + P_7,
\end{align*}
which finishes the proof.
\end{proof}

\begin{lemma}
\label{lemma:cyan}
For $\mu_1 = \dfrac{ \beta  c}{ 2 \sqrt A }$, $\nu_1 = \dfrac { \frac{ 4 \beta ( 1-  \alpha)} n + c }{ 2 \sqrt A}$, and $\omega_1 = \sqrt{ \dfrac{ 2 ( 1 - \alpha)} n }$, define $$P_8(\phi) := - ( \mu_1 + \nu_1 \phi)^2  +(\omega_1 \phi)^2 + \phi_t.$$
If $A > 0$, then  $P_8 > 0$ implies $P_6 > 0$ for any $(x,t)$.
\end{lemma}

\begin{proof}[Proof of Lemma 7]
Recall that
\[ P_6 = A e^{2u} + \left[ \left  (\dfrac{ 4 \beta ( 1- \alpha)}n + c \right ) \phi + \beta c \right] e^u + \dfrac{ 2  ( 1- \alpha)} n \phi^2 + \phi_t.\]
Since $A > 0$, we use the fact that $x^2 + xy \geq -\frac{1}{4} y^2$ in the form
\[A e^{2u} + \left [ \left  (\dfrac{ 4 \beta ( 1- \alpha)}n + c \right ) \phi + \beta c \right ] e^u \geq - \dfrac{  \left [ \left (  \dfrac{ 4 \beta ( 1 - \alpha) }n + c \right ) \phi + \beta c \right ]^2 }{ 4 A}. \]
This gives that
\begin{align*}
 P_6 & \geq - \dfrac{  \left [ \left (  \dfrac{ 4 \beta ( 1 - \alpha) }n + c \right ) \phi + \beta c \right ]^2 }{ 4 A} + \dfrac{ 2 ( 1- \alpha)} n \phi^2 + \phi_t \\
& = - \left [\dfrac{ \beta c}{ 2 \sqrt A}  +  \dfrac 1 {2 \sqrt A} \left( \dfrac{4 \beta ( 1 - \alpha)} n + c \right )  \phi \right ]^2 + \left ( \phi \sqrt{ \dfrac{ 2 ( 1- \alpha)} n} \right )^2 + \phi_t
\end{align*}
as desired.
\end{proof}

\begin{lemma}
\label{lemma:magenta}
If condition (iii) of Theorem \ref{theorem:white} holds, then there always exists some $\eps' > 0$ such that $A > 0$ and $\nu_1^2 < \omega_1^2$.
\end{lemma}

\begin{proof}[Proof of Lemma 8]
We first want to show that $A(\varepsilon') > 0$ for some $\varepsilon' > 0$. We will show that $A(0) > 0$, and since $A$ is a continuous function of $\varepsilon'$, this implies that $A(\varepsilon') > 0$ for some $\varepsilon' > 0$.

We have that
\begin{align*}
	A(0) & = \frac{2 \beta^2 (1-\alpha)}{n} - \frac{n\left(c + \frac{4\beta (1-\alpha)}{n}\right)^2}{8(1-\alpha -0)} \\
		& = \frac{16 \beta^2 (1-\alpha)^2 - (cn + 4 \beta (1-\alpha))^2}{8 n (1-\alpha)} \\		& = \frac{-c^2n^2 - 8 \beta c n (1-\alpha)}{8n(1-\alpha)}.
\end{align*}
	It follows from (iii) that
 \[-8 < -4-2\sqrt{2} < \frac{cn}{\beta(1-\alpha)},\]
	which rearranges to give $c^2 n^2 + 8 \beta c n (1-\alpha)< 0$. Thus $A(0) > 0$, and so there exists some $\varepsilon' > 0$ such that $A(\varepsilon') > 0$.

\vspace{12pt}
Next we want to show that $\nu_1^2 < \omega_1^2$ for some $\varepsilon' > 0$, where 
 	\begin{align*}
 		\nu_1 & = \dfrac { \frac{ 4 \beta ( 1-  \alpha)} n + c }{ 2 \sqrt A}, \\
 		\omega_1 & = \sqrt{\frac{2(1-\alpha)}{n}}.
 	\end{align*}
 	Since $\nu_1$ and $\omega_1$ are continuous functions of $\varepsilon'$, if we can show that $\nu_1^2 < \omega_1^2$ for $\varepsilon' = 0$, then it must be that $\nu_1^2 < \omega_1^2$ for some $\varepsilon' > 0$.
 	
 	When $\varepsilon'=0$, $\nu_1^2 < \omega_1^2$ is equivalent to
	 	\[ \frac{c^2n^2+8\beta c n (1-\alpha) + 16(1-\alpha)^2 \beta^2}{-c^2n^2-8\beta c n (1-\alpha)} < 1. \]
	Restriction (iii) implies 
	\[ -4-2\sqrt{2} < \frac{cn}{\beta(1-\alpha)} < -4 + 2\sqrt{2},\]
	which leads to
	\[ \frac{c^2n^2}{\beta^2(1-\alpha)^2} + \frac{8 cn}{\beta(1-\alpha)} + 8 < 0.\]
	This is equivalent to
	\[c^2n^2 + 8 \beta c n (1-\alpha) + 16(1-\alpha)^2 \beta^2  < - (c^2 n^2 + 8 \beta c n (1-\alpha)) \]
	and therefore $\nu_1^2 < \omega_1^2$ for $\varepsilon'=0$.
	\end{proof}

\begin{lemma}
\label{lemma:brown}
Suppose $R \geq 1$ is a constant and $\rho: M^n \rightarrow \mathbb R$ is a function that satisfies 
\begin{equation*}
\rho(x) \geq 0, \qquad |\nabla \rho(x)| \leq 1, \qquad \Delta \rho \leq \dfrac{ c_1} {\rho},
\end{equation*}
for some constant $c_1 > 0$. Define
\begin{equation}
\label{eq:metroid}
\psi(x) := k \dfrac{ R^2 + \rho^2}{(R^2 - \rho^2)^2},
\end{equation}
then for $k$ sufficiently large, $\psi(x)$ satisfies $P_7 > 0$.
\end{lemma}

\begin{proof}[Proof of Lemma 9]
Let
\[ \Psi(x) := \dfrac{ R^2  +\rho
^2}{(R^2 - \rho^2)^2}, \]
so that $\psi = k \Psi$. We claim that $\Psi$ satisfies
\begin{equation}
\label{eq:tetris}
|\nabla \Psi|^2 \leq 18 \Psi^3, \qquad \Delta \Psi \leq c_2 \Psi^2,
\end{equation}
where $c_2$ depends only on $c_1$.

Indeed, we can compute
\begin{align*}
\nabla \Psi  &= \nabla \rho \left ( \dfrac{ 6 \rho R^2 + 2 \rho^3 }{ (R^2 - \rho^2)^3} \right ), \\
|\nabla \Psi|^2 & \leq 4 \rho^2 \dfrac{ ( 3 R^2 + \rho^2)^2}{(R^2 - \rho^2)^6} \leq  18 \Psi^3,\\
\end{align*}
and
\begin{align*}
\Delta \Psi & = \Delta \rho \left ( \dfrac{ 6 \rho R^2 + 2 \rho^3 }{(R^2 - \rho^2)^3} \right ) + |\nabla \rho|^2 \left ( \dfrac{6 R^4 + 36 \rho^2 R^2 + 6 \rho^4}{(R^2 - \rho^2)^4} \right ) \\
& \leq 6 c_1 \dfrac{ R^2 + \rho^2}{(R^2 - \rho^2)^3} + 18 \dfrac{ (R^2 + \rho^2)^2}{(R^2 -\rho^2)^4} \\
& \leq (6 c_1 + 18 ) \Psi^2.
\end{align*}

Recall that
\[
P_7 = \dfrac{ 4 \alpha ( 1- \alpha)} n \psi |\nabla u|^2 - 2 \nabla u \cdot \nabla \psi + \dfrac{ 2 \eps' }n \psi^2 - \Delta \psi. \]
Completing the square gives us
\[
P_7 \geq \dfrac{ 2 \eps'}n \psi^2 - \Delta \psi - \dfrac{ n }{ 4 \alpha ( 1 - \alpha) \psi} |\nabla \psi|^2.
\] 
By \eqref{eq:tetris}, we know that:
\begin{gather*}
\dfrac{ \eps'}n k^2 \Psi^2 \geq \dfrac{ \eps' k}{c_2 n} k \Delta \Psi, \\
\dfrac{ \eps' }n k^2 \Psi^2 \geq \dfrac{ \eps' k}{18 n} \cdot \dfrac{ k^2 |\nabla \Psi|^2}{k \Psi},
\end{gather*}
so if 
\begin{equation*}
k >  \max \left ( \dfrac{c_2 n}{\eps'}, \dfrac{ 18 n^2 }{ 4 \alpha ( 1 - \alpha) \eps' }
\right ) ,
\end{equation*}
then it immediately leads to that $P_7 > 0$.

\end{proof}

\subsection{Proof of Theorem \ref{theorem:white}} We are now ready to prove  our Theorem \ref{theorem:white}. 
\begin{proof} Fix a point $p \in M$, let $r= r(x) := d(x,p)$, where $d(\cdot, \cdot)$ denotes the geodesic distance in $M$. 
We define the Harnack quantity $h$ on the geodesic ball $B_R(p):= \{ x \in M \mid d(x,p) < R \}$. $h$ depends on the positive constants  $\eps, \eps', k$,  $R$ and is defined as follows:
\begin{align*}
h(x,t) & = \Delta u + \alpha |\nabla u|^2 + \beta e^u + \phi(t) + \psi(x), \\
\phi &= \phi(t) : = \dfrac{ \mu_2 \left ( \dfrac 1 { \nu_2 - (\omega_2 - \eps)} e^{2 \mu_2 (\omega_2 - \eps) t}  - \dfrac 1 { \nu_2 + ( \omega_2 - \eps)} \right )}{ 1 - e^{ 2 \mu_2 ( \omega_2 - \eps)t }}, \\
\psi &=  \psi(x) := k \dfrac{R^2 + r^2}{(R^2 - r^2)^2},
\end{align*}
with $\mu_2, \nu_2, \omega_2,$ and $A$ defined as in Lemma \ref{lemma:cyan} and the beginning of Section 3.2. Fix $R> 1$. Let $\varepsilon, \varepsilon'$ and $k$ be positive constants to be chosen later. Note that $h$ is $C^1$ in $t$ and $C^2$ in $x$, except possibly for those $x$  in the cut locus $\mathcal C(p)$. We will show that we can choose $\eps, \eps'$, and $k$ so that $h(x,t) > 0$ for all $x, t$. Assume for the sake of a contradiction that $h(x,t) \leq 0$ for some $x,t$.

Let $t_1$ be the first time $t$ such that  $\inf_{x \in B_R(p)} h(x,t) = 0$. Since $\lim_{t \rightarrow 0^+} h(t) = \infty$ by Lemma \ref{lemma:yellow},  it follows that $t_1$ exists.
Note also that $\psi(x) \rightarrow \infty$ as $r=d(x, p)$ approaches $R$, so the infimum of $h$ is attained inside $B_R(p)$; let $(x_1, t_1)$ be such a point, so that $h(x_1,t_1) = 0$. Now we split into cases based on whether  or not $x_1$ is in the cut locus $\mc C (p)$.

Case 1: Suppose that $x_1 \notin \mathcal C(p)$, so that $\psi(x)$ is twice differentiable at $x_1$. Then by Lemmas \ref{lemma:red} and  \ref{lemma:blue} (from subsection 2.2) and \ref{lemma:orange} (subsection 3.2) we have that
\begin{equation*}
 0 > h_t - \Delta h - 2 \nabla h \cdot \nabla u - P_1 h \geq P_2 + P_3 + P_4 \geq P_6 + P_7. 
\end{equation*}
By Lemma \ref{lemma:magenta}, we can choose $\eps' > 0$ small enough such that $A > 0$ and $\nu^2 < \omega^2$; then, since $\phi$ is the same as the one defined as in Lemma \ref{lemma:yellow}, it follows by Lemmas \ref{lemma:yellow} and \ref{lemma:cyan} that we can choose $\varepsilon$ small enough so that $P_6 > 0$.

Note that $\psi$ takes the form of \eqref{eq:metroid}, with the distance function $\rho(x) = r(x) = d(x,p)$. We have that $r \geq 0$ and $|\nabla r|^2 = 1 $; furthermore, by the Laplacian comparison theorem  we have that $\Delta r \leq \dfrac{n-1}r$. Thus we can apply Lemma \ref{lemma:brown} and choose $k$ sufficiently large such that $P_7 > 0$ as well, which leads to a contradiction.

Case 2: Suppose that $x_1 \in \mathcal C(p)$. We apply Calabi's trick. Let $ \delta \in \left ( 0 , \frac{d(x_1, p)}{2} \right )$ be a positive constant, and let $\gamma(t)$ be any length-minimizing geodesic from $p$ to $x_1$. Define $p_\delta := \gamma(\delta)$, so that $ x_1 \notin \mathcal C(p_\delta)$, and define
\begin{gather*}
r_\delta(x) := d(x, p_\delta) + \delta, \\
\psi_\delta(x) := k \dfrac{ R^2 + r_\delta^2}{(R^2 - r_\delta)^2}, \\
h_\delta := \Delta u + \alpha |\nabla u|^2 + \beta e^u + \phi + \psi_\delta.
\end{gather*}
Note that by the triangle inequality,
\[ r_\delta (x)= d(x,p_\delta) + d(p_\delta, p) \geq r(x), \]
with equality at $x = x_1$.  Since $\psi$ is an increasing function of $r$, it follows that $\psi_\delta(x) \geq \psi(x)$ with equality at $x=x_1$. This implies that $(x_1, t_1)$ is still the first time and place where $h_\delta(x,t) = 0$. Furthermore, $h_\delta$ is now $C^2$ at $(x_1, t_1)$ so applying Lemmas \ref{lemma:red}, \ref{lemma:blue}, \ref{lemma:orange}, \ref{lemma:yellow}, and \ref{lemma:cyan} gives that
\[ 0 > P_7. \]
Note that clearly $r_\delta \geq 0$ and $|\nabla r_\delta| \leq 1$, and at $x_1$ we get
\[ \Delta r_\delta = \Delta (d(x_1, p_\delta)) \leq \dfrac{n-1}{d(x_1, p_\delta)} = \dfrac{n-1}{r(x_1) - \delta} \leq \dfrac{2(n-1)}{r(x_1)}, \]
since we assumed that $ \delta \leq \frac{ r(x_1)}{2}$. Therefore applying Lemma \ref{lemma:brown} gets us a contradiction in this case as well.

This shows   that $h(x,t) > 0$ for all $x,t$. Since $h$ varies continuously as a function of $R, \eps, \eps'$, we can take the limit $R \rightarrow \infty$ to get $\psi \rightarrow 0$. Then by taking $\eps, \eps' \rightarrow 0$, we get that $\phi \to \phi_1$ and so
\[ \Delta u + \alpha |\nabla u|^2  + \beta e^u + \dfrac{ \mu_2 \left ( \frac 1 { \nu_2 - \omega_2} e^{ 2 \mu_2 \omega_2 t } - \frac 1 { \mu_2 + \omega_2 } \right )}{ 1 - e^{2 \mu_2 \omega_2 t}} \geq 0, \]
with
\begin{align*}
\mu_2 & =  \beta c \sqrt{ \frac{ 2 ( 1- \alpha)} {c ( - cn - 8 \beta ( 1- \alpha))} }, \\
\nu_2 & = \left ( \frac{ 4 \beta ( 1- \alpha)} n + c \right ) \cdot \sqrt { \frac{ 2  ( 1- \alpha)} { c ( - cn  - 8 \beta ( 1- \alpha))} }, \\
\omega_2 & =  \sqrt{  \frac { 2 ( 1 -\alpha)} n },
\end{align*}
which finishes the proof.
\end{proof}

%%%%%%%%%%%%%%%%%%%%%%%%%%%
\section{Applications}%%%%%%%%%%%%%%%%%
%%%%%%%%%%%%%%%%%%%%%%%%%%%
In this section, we  shall derive two applications of our differential Harnack Estimates.

\subsection{Bounds on the Wave Speed of Traveling Wave Solutions}

The first such application shows that our Harnack inequality can be used to prove an interesting fact about traveling wave solutions to Fisher's equation. In particular we look at traveling plane waves, i.e. solutions to \eqref{eq:fisher} of the form
\[ f(x,t) = v(z) := v(x + \eta t \hat a), \]
for some function $v: \mathbb R^n \rightarrow \mathbb R$ and some wave direction $\hat a \in \mathbb R^n, |\hat a| = 1$ and wave speed $\eta > 0$. For $n=1$, these solutions were first studied by Fisher in \cite{fisher37} (also see \cite{kpp37, sherratt98}) and were considered by him to be a natural model for propagation of mutations. He was able to show that if $n=1$ and $\displaystyle \lim_{t \rightarrow - \infty} f(x,t) = 0$, then it must be that $\eta \geq 2 \sqrt c$.

We will show a weaker bound that generalizes to higher dimensions.
\begin{theorem}
\label{theorem:emerald}
Let $f(x,t) = v(x + \eta t \hat a)$ be a traveling plane wave solution to \eqref{eq:fisher}, with wave speed $\eta$ and wave direction $\hat a$. Suppose that
\begin{equation}
\label{eq:dk}
\lim_{\substack{ x = k \hat b, \\ k \rightarrow \infty}} v(x) = 0 \quad \text{ for some direction }  \hat b \in \mathbb R^n, \  |\hat b | \neq 0.
\end{equation}

Then
\[ \eta \geq \begin{cases} \sqrt{ ( 3 - \sqrt 3) c }, & n = 1, \\ \sqrt{2c}, & n =2, \\ \sqrt{ ( 7 - 3 \sqrt 3) c }, & n = 3. \end{cases} \]
\end{theorem}

\begin{remark}
When $n = 1$, $\eta \geq 2\sqrt{c}$ is both a necessary and sufficient condition for the existence of traveling wave solutions. The same condition is sufficient in any higher dimension, but it is not known (at least to us) if it is necessary as well. Our bounds above give a weaker necessary wave speed in dimension two and three.

\end{remark}

\begin{remark}
In the proof below we have not used the fact that the traveling wave $v$ approaches $1$ in some direction. Although we were ourselves unsuccessful, the authors would like to encourage an attempt to use this additional restriction to obtain a better bound on the wave speed $\eta$.

\end{remark}

\begin{lemma}
\label{lemma:silver}
For any function $v(z)$ and any $\eta$ that satisfy the conditions of Theorem \ref{theorem:emerald}, and for any $\alpha, \beta$ that satisfy (i), (ii), and (iii) as in Theorem \ref{theorem:white}, we have 
\begin{equation*}
\eta^2 \geq M' := 4 ( 1- \alpha)\left [ (c  - \phi(t) ) - (\beta + c) v(z) \right ] ,
\end{equation*}
for all $x,t$, where $\phi(t)= \frac{\mu \left( \frac{1}{\nu - \omega} e^{2 \mu \omega t} - \frac{1}{\nu + \omega} \right)}{\left(1 - e^{2 \mu \omega t} \right)} $ (which appears as $\phi_1(t)$ in the statement of Theorem \ref{theorem:white}).
\end{lemma}

\begin{proof}
Since Fisher's equation is spherically symmetric, we may assume without loss of generality that $\hat a = \hat x_1 = (1,0,0, \ldots, 0)$. Therefore
\[ f(x,t) = v(x_1  + \eta t, x_2, \ldots, x_n)= v(z_1,z_2,\ldots,z_n)=v(\hat z).\]
It then follows from \eqref{eq:fisher} that $\left( \text{where } \partial_i := \frac{\partial}{\partial z_i} \right)$
\[ \eta \partial_1 v = \Delta v + c v ( 1 - v). \]

Combining this with Theorem \ref{theorem:white} gives that:
\begin{gather*}
\Delta ( \log v ) + \alpha |\nabla (\log v)|^2 + \beta v + \phi \geq 0; \\
\dfrac{ \Delta v}{v} - ( 1 - \alpha ) \dfrac{ |\nabla v|^2 }{ v^2 } + \beta v  + \phi \geq 0; \\
\dfrac{ \eta \partial_1  v-  c v ( 1- v) }{ v} -  ( 1- \alpha ) \dfrac{ |\nabla v |^2 } {v^2 } + \beta v + \phi \geq 0; \\
( 1 -\alpha) \dfrac{ \sum_{i=2}^n (\partial_i v)^2 }{v^2} + ( 1- \alpha) \dfrac{ (\partial_1 v)^2 }{v^2 } - \eta \dfrac{ \partial_1 v}{v } - ( \beta + c ) v +  (c - \phi ) \leq 0.
\end{gather*}
It follows from standard Cauchy-Schwarz that
\[ - \dfrac{ \eta^2}{ 4 ( 1- \alpha)} -  ( \beta + c ) v + ( c -\phi) \leq 0, \]
hence
\[ \eta^2 \geq 4 ( 1- \alpha ) [ ( c - \phi) - ( \beta + c )v ], \]
as desired.
\end{proof}

\begin{lemma}
\label{lemma:crystal}
Assume that $v(x) \to 0$ along some path, as in \eqref{eq:dk}. Then for any $\eps_3 > 0$ there exists $(x_3, t_3)$, possibly depending on $n, \alpha, \beta$, and $c$, such that at $(x_3, t_3)$
\begin{equation*}
M' > M'' - \dfrac{\eps_3}3,
\end{equation*} 
where
\begin{equation*}
M'' := 4( 1- \alpha) \left (  c - \dfrac{ - \mu }{ \nu + \omega} \right ).
\end{equation*}
\end{lemma}
\begin{proof}
Fix $\eps_3 > 0$. Note that
\[ \lim_{t \rightarrow \infty} \phi(t) = \dfrac{ - \mu}{\nu + \omega}. \]
Choosing $t \geq t_3$ large enough gives 
\[ \left | \phi(t_3) - \dfrac{ - \mu }{ \nu+ \omega} \right | < \dfrac{ \eps_3}{24(1-\alpha)},
\]
so that
\[ 4 ( 1 -\alpha)( c - \phi ) > 4 ( 1 -\alpha)\left ( c - \dfrac{ - \mu}{\nu+ \omega} \right ) - \dfrac{ \eps_3}6. \]

Having fixed $t_3$, we then set $x_3 := - \eta t_3 \hat a + \lambda \hat b$ with $\lambda$ sufficiently large. Then by \eqref{eq:dk} it follows that
\[ \left | v -  0 \phantom{\dfrac 12} \! \! \right | < \dfrac{\eps_3}{24 (  1 - \alpha ) } \dfrac 1 { |\beta  + c |}, \]
\[-  4 ( 1 - \alpha)( \beta + c ) v > 0 - \dfrac { \eps_3}6. \]
Therefore
\[ M' = 4  (1 - \alpha) [ ( c - \phi) - ( \beta + c )\phi] > M'' - \dfrac{\eps_3}3.\]
\end{proof}
\begin{remark}
Note that \eqref{eq:dk} can be weakened; it suffices to have $\displaystyle \lim_{z \rightarrow \infty} v(z) = 0$ along some path that goes to infinity.
\end{remark}

\begin{lemma}
\label{lemma:gold}
If $n \leq 3$, and $\beta = - \dfrac {cn ( 1+ \alpha)} { 4 \alpha^2 - 4 \alpha + 2n}$, and $0 < \alpha < \alpha_0(\eps_3) $ is sufficiently close to 0, then conditions (i), (ii), and (iii) are satisfied, and 
\begin{equation*}
M'' > M''' - \dfrac{\eps_3}3,
\end{equation*}
where
\begin{equation*}
M''' := M'''(n) = 2c \left ( \dfrac{ n- 4 + 2\sqrt{4n-n^2} }{ n - 2 + \sqrt{4n-n^2} } \right ).
\end{equation*}
\end{lemma}

\begin{proof}
(i) and (ii) are clearly satisfied by construction. And note that (iii) is equivalent to
\[ - \dfrac{ 2  +\sqrt 2 } 4 < \dfrac{ \beta ( 1 - \alpha)} {cn } < - \dfrac{ 2 - \sqrt 2 }4. \]
But the quantity in the middle varies continuously with $\alpha$ near $\alpha = 0$, so it suffices to check it at $\alpha=0$, where we indeed have
\[ - \dfrac{ 2 + \sqrt 2} 4 < - \dfrac{ 1}{2n} < - \dfrac{ 2 - \sqrt 2 }4 , \]
which holds for all $n \leq 3$, so there must exist some $\alpha_0$ sufficiently small such that (iii) holds for all $\alpha < \alpha_0$.

Next, we compute $M''$:

\begin{gather*}
M''  = 4  (1 -\alpha) \left ( c - \dfrac{ - \mu}{\nu + \omega} \right ) \\
= 4 ( 1 - \alpha ) \left ( c + \dfrac{ \dfrac{ \beta c }{ 2 \sqrt A}}{\dfrac 1 { 2 \sqrt A}  \left ( \dfrac{ 4 \beta ( 1 -\alpha)} n + c \right ) + \sqrt{ \dfrac{ 2 ( 1 -\alpha)} n } } \right ) \\
= 4 ( 1 -\alpha ) \left ( c + \dfrac{ \beta c } { 	\left (  c + \dfrac{ 4 \beta  (1 -\alpha)} n \right ) + \sqrt{ \dfrac{ 8 A ( 1 - \alpha)} n } } \right ).
\end{gather*}
Here $A = A(\eps' = 0 )$, so that
\[ \dfrac{ 8 A ( 1 -\alpha)} n =  \dfrac{ 16 \beta^2  ( 1- \alpha)^2 }{n^2 } - \left (c+ \dfrac {4 \beta ( 1- \alpha)} n \right)^2 = c^2 \left ( - 1 - \dfrac{ 8 \beta ( 1 - \alpha)} {cn} \right ). \] 
This gives
\begin{gather*}
M'' = 4  (1 -\alpha) c \left ( 1 + \dfrac{ \beta /c  }{ 1 + \dfrac{ 4 \beta ( 1 -\alpha)} {cn} + \sqrt{ -1  - \dfrac{ 8 \beta ( 1 -\alpha)} {cn} } } \right ).
\end{gather*}
Again, this involves only $(1-\alpha)$ and $\beta$, both of which are continuous at $\alpha = 0$, where we have $\beta = -c/2$, so
\[ M'' = 4 c \left ( 1 + \dfrac{ - 1/2 }{ 1 - \dfrac 2n + \sqrt{ - 1 + \dfrac 4n }} \right ) = 2 c \left ( 2 - \dfrac n { n - 2 + \sqrt{4n - n^2}} \right ) = M'''. \]

Hence for $\alpha$ sufficiently close to $0$ we can get $|M'' - M'''| < \eps_3/3$, which gives us the desired conclusion.
\end{proof}
\begin{proof}[Proof of Theorem]
Fix a solution $f(x,t)  =v(x + \eta t \hat a) $ of \eqref{eq:fisher} which also satisfies \eqref{eq:dk}, and fix a $\eps_3 > 0$.

Let $\alpha < \alpha_0$ and $\beta = - \dfrac{ c}{ 2 ( 1 - \alpha)}$, so that (i), (ii), (iii) are satisfied (by Lemma \ref{lemma:gold}). Applying Lemma \ref{lemma:silver} then gives that $\eta^2 \geq M$ for all $x,t$.

Applying Lemma \ref{lemma:crystal}, we find a pair $(x_3, t_3)$ such that $M' > M'' - \eps_3/3$. Then applying Lemma \ref{lemma:gold} again, we have that $M'' > M'''- \eps_3/3$ so that
\[ \eta^2 > M''' - \eps_3. \]
However, note that $M'''$ depends only on $n$. Hence we send $\eps_3 \rightarrow 0$, to get that
\[ \eta^2 \geq M'''(n) = \begin{cases} c ( 3 - \sqrt 3), & n = 1, \\
2c ,  & n = 2, \\
c ( 7 - 3 \sqrt 3), & n = 3,
\end{cases} \]
as desired.
\end{proof}

\subsection{Classical Harnack Inequality} 

In this subsection, we integrate our differential Harnack estimates along a space-time curve to derive classical Harnack inequalities. We further assume that $M$ is closed, and that $f(x,t) < 1$ for all $x,t$.

\begin{theorem}
\label{theorem:ruby} Let $M$ be a closed Riemannian manifold with non-negative Ricci curvature, and $0 < f < 1$ be a bounded positive solution to Fisher's equation. Let $\alpha$ and $\beta$ satisfy the conditions of Theorem \ref{theorem:black}. Furthermore, if $\alpha \leq n /4$, then there will always exist $\beta$ such that $\beta + c \geq 0$ in addition to the constraints of Theorem \ref{theorem:black}. For such an $\alpha$ and $ \beta$, 

\begin{itemize}
\item[(i)]  if  $8 \beta (1- \alpha)+ cn < 0$, then we have

\[   \frac{ f(x_2, t_2 )}{f(x_1, t_1) } \geq  \left( \frac{1 - e^{-ct_2}}{1-e^{-ct_1}} \right)^{\frac{8 \beta^2 (1-\alpha)}{c^2 n + 8 \beta c (1-\alpha)}}  \exp\left( - \dfrac { d(x_1, x_2)^2 } { 4 ( 1 - \alpha) (t_2  -t_1)} \right); \]

\item[(ii)]  if  $8 \beta (1-\alpha) + cn > 0$, $t_2>t_1>T_2$, then we have 

\[  \frac{ f(x_2, t_2 )}{f(x_1, t_1) } \geq   \left[ \frac{\left(1 + \frac{8\beta(1-\alpha)}{cn}\right)e^{-c(t_2-T_2)}+1}{\left(1 + \frac{8\beta(1-\alpha)}{cn}\right)e^{-c(t_1-T_2)}+1} \right]^{\frac{8 \beta^2 (1-\alpha)}{c(cn + 8 \beta (1-\alpha))}} \exp\left( - \dfrac { d(x_1, x_2)^2 } { 4 ( 1 - \alpha) (t_2  -t_1)} \right); \]

\item[(iii)]  if  $8 \beta (1-\alpha) + cn = 0$, $t_2>t_1>T_2$, then we have 

\[   \frac{ f(x_2, t_2 )}{f(x_1, t_1) } \geq   \exp\left[ - \dfrac { \beta } { c} \left(e^{-c(t_2-T_2)}- e^{-c(t_1-T_2)}\right) \right] \exp\left( - \dfrac { d(x_1, x_2)^2 } { 4 ( 1 - \alpha) (t_2  -t_1)} \right). \]

\end{itemize}

\end{theorem}

\begin{proof}[Proof of Theorem \ref{theorem:ruby}]
Let $f(x,t)$ solve $f_t = \Delta f + c f ( 1- f)$, and $u = \log f$. Fix points $(x_1, t_1), (x_2, t_2)$ and let $\gamma: [t_1, t_2] \rightarrow M^n$ be an arbitrary space-time path connecting them, i.e. $\gamma(t_1) = x_1, \gamma(t_2) = x_2$.

Let $v(t) := u(\gamma(t), t)$ be the value of $u$ along $\gamma$. We compute
\[ v'(t) = u_t + \nabla u \cdot \dfrac { d \gamma } {dt}. \] Using the time evolution for $u_t = (\log f)_t = \dfrac{f_t}f$, this is equal to
\[ v'(t) = \Delta u  + | \nabla u|^2 + c ( 1 -e^u) + \nabla u \cdot \dfrac{ d \gamma}  {dt}. \]
Applying the Harnack inequality gives
\begin{align*}
v'(t) & \geq ( 1-  \alpha) |\nabla u|^2 + (c-\phi) - ( \beta + c ) e^u  +\nabla u \cdot \dfrac{ d \gamma } {dt}.
\end{align*} 
By assumption, $f < 1$ and $\beta + c \geq 0$. This implies
\[ - (\beta + c ) e^u \geq -(\beta + c), \]
so defining $\widetilde{\phi}(t) = - \beta - \phi(t)$, we then get
\begin{align*}
v'(t) & \geq ( 1-  \alpha) |\nabla u|^2 + (c-\phi) - ( \beta + c )  +\nabla u \cdot \dfrac{ d \gamma } {dt} \\
 & = -\beta - \phi + ( 1-  \alpha) |\nabla u|^2  +\nabla u \cdot \dfrac{ d \gamma } {dt} \\
 & = \widetilde{\phi}(t) + ( 1-  \alpha) |\nabla u|^2  +\nabla u \cdot \dfrac{ d \gamma } {dt}, \\
 v'(t) & \geq    \widetilde{\phi}(t) - \dfrac 1 {  4 (1-\alpha)} \left | \dfrac {d \gamma} {dt} \right |^2. 
\end{align*}
Integrating in time, we get
\[ u(x_2, t_2) - u(x_1, t_1) =  v(t_2) - v(t_1) = \int_{t_1}^{t_2} v'(t) dt  \geq \int_{t_1}^{t_2} \widetilde{\phi}(t) dt - \dfrac 1  {4  ( 1 -\alpha)} \int_{t_1}^{t_2} \left | \dfrac{ d \gamma}{dt} \right |^2 dt. \]
Since $\gamma$ was chosen to be an arbitrary path, we can choose it to be the path minimizing $\int | \gamma' |^2$, which is the minimizing geodesic between the two endpoints. The integral thus becomes
\[ \int_{t_1}^{t_2} | \gamma' |^2 dt =  \dfrac { d(x_1, x_2)^2 } {t_2 - t_1}. \]
Thus the space-time Harnack is given by
\[ \log \left ( \dfrac { f(x_2, t_2 )}{f(x_1, t_1)} \right )  = u(x_2, t_2) - u(x_1, t_1) \geq \int_{t_1}^{t_2}  \widetilde{\phi}(t) dt - \dfrac { d(x_1, x_2)^2 } { 4 ( 1 - \alpha) (t_2  -t_1)}. \]

We compute the definite integral, dividing into three cases. 
First we deal with the case $8 \beta (1- \alpha)+ cn < 0$. In this case we have that 
\[\phi(t) = \frac{\left( \frac{\beta c n}{c n + 8 \beta (1 - \alpha)} \right) e^{-c t} - \beta}{1 - e^{-c t}}, \]
and
\[ \widetilde{\phi}(t) = \left ( \beta e^{-ct} - \dfrac{ \beta c n e^{-ct} }{cn + 8 \beta ( 1 -\alpha)} \right ) \dfrac{ 1 }{1 -  e^{-ct}} = \beta \cdot \dfrac {8 \beta ( 1- \alpha)}{cn + 8 \beta ( 1 - \alpha)} \cdot \dfrac {e^{-ct}}{1 - e^{-ct}}. \]
Then we can explicitly integrate 
\begin{align*}
\int_{t_1}^{t_2} \widetilde{\phi}(t) \; dt & = \dfrac{\beta}c   \left( \frac{8 \beta (1-\alpha)}{c n + 8 \beta (1-\alpha)} \right)\log\left[ \frac{1 - e^{-ct_2}}{1-e^{-ct_1}} \right].
\end{align*}
Therefore we get that 
\[\exp \left( \int_{t_1}^{t_2} \widetilde{\phi}(t) \; dt \right) = \left( \frac{1 - e^{-ct_2}}{1-e^{-ct_1}} \right)^{\frac{8 \beta^2 (1-\alpha)}{c^2 n + 8 \beta c (1-\alpha)}}\] and the claim follows.

Second, we deal with the case $8 \beta (1-\alpha) + cn > 0$. Then for $t > T_2$ (recall that $T_2$ is a constant) we have that
\[ \phi(t) = \dfrac{ - \beta c n e^{c (t-T_2)} -\beta c n }{ c n+  8 \beta ( 1- \alpha) + c n e^{ c (t-T_2)} } ,\]
and so
\begin{align*}
\widetilde{\phi}(t) = - \beta - \phi(t) = \frac{- 8 \beta^2 (1-\alpha) e^{-c(t-T_2)}}{(8 \beta (1- \alpha) + cn) e^{-c(t-T_2)} + cn} .
\end{align*}
If we let $B =-8\beta^2(1-\alpha)$ and $D = cn + 8\beta(1-\alpha)$, then we get that
\[ \widetilde{\phi}(t) = \frac{B e^{-c(t-T_2)}}{De^{-c(t-T_2)}+cn}.\]
We can integrate
\begin{align*}
	\int_{t_1}^{t_2} \widetilde{\phi}(t)\; dt  =  \left( \frac{8 \beta^2 (1-\alpha)}{c^2 n + 8 \beta c (1-\alpha)} \right)\log \left( \frac{(8\beta (1-\alpha) + cn )e^{-c(t_2-T_2)}+cn}{(8 \beta (1- \alpha) + cn) e^{-c(t_1-T_2)}+cn} \right).
\end{align*}
Therefore
\[\exp \left( \int_{t_1}^{t_2} \widetilde{\phi}(t) \; dt \right) = \left[ \frac{\left(1 + \frac{8\beta(1-\alpha)}{cn}\right)e^{-c(t_2-T_2)}+1}{\left(1 + \frac{8\beta(1-\alpha)}{cn}\right)e^{-c(t_1-T_2)}+1} \right]^{\frac{8 \beta^2 (1-\alpha)}{c^2 n + 8 \beta c (1-\alpha)}}\] as claimed in the statement of Theorem \ref{theorem:ruby}. 

In the last case that $8 \beta (1- \alpha)+ cn = 0$, we have that
\[ \phi(t) = \dfrac{ - \beta e^{c (t-T_2)} -\beta }{  e^{ c (t-T_2)} } ,\]
and so
\begin{align*}
\widetilde{\phi}(t) = - \beta - \phi(t) = \frac{ \beta}{ e^{c(t-T_2)} } .
\end{align*}
Therefore
\[\exp \left( \int_{t_1}^{t_2} \widetilde{\phi}(t) \; dt \right) = \exp\left[ -\dfrac { \beta } { c} \left(e^{-c(t_2-T_2)}- e^{-c(t_1-T_2)}\right) \right]\] as desired.

To finish the proof of our theorem we need to show that we can choose $\beta + c \geq  0$, i.e. $\beta \geq  -c$. We have the constraint (ii): 
\[ \beta \leq \dfrac{ - c n ( 1 + \alpha)}{4 \alpha^2 - 4 \alpha + 2 n}, \]
 so we need to have
\[ -c \leq \beta \leq \dfrac{ - cn ( 1 + \alpha)}{4 \alpha^2 - 4 \alpha + 2n}. \]
Note that since $0 < \alpha < 1$, we have $4 \alpha^2 - 4 \alpha + 2n \geq -1  +2n \geq 1$; thus it remains to choose $\alpha$ so that
\[ -(4 \alpha^2 - 4\alpha + 2n) \leq - n ( 1 + \alpha), \]
which simples to 
\[ \alpha \leq n/4. \]
This is  automatically true if $n \geq 4$, which means we can choose an $\alpha$ we wish, and then there will be at least one $\beta$ that satisfies all the constraints including $\beta + c \geq 0$.
\end{proof}
\vspace{12pt}

\begin{remark}
Note that $\displaystyle \lim_{t \to \infty} \phi(t) = - \beta$, and $\displaystyle \lim_{t \rightarrow \infty} \widetilde{\phi}(t) = 0$. Thus, as $t_1, t_2 \to \infty$, the estimate approaches the classical Li-Yau Harnack \cite{ly86}.
\end{remark}

\begin{remark}
In the compact case we obtain a good bound as $t_1$ and $t_2$ get large. In the complete noncompact case, one can still integrate along space-time curves to obtain an inequality, but the estimate degenerates when time becomes large. 
\end{remark}

\bibliographystyle{plain}
\bibliography{bio}
\end{document}